\documentclass[12pt, reqno]{amsart}
\usepackage{amsmath, amsthm, amscd, amsfonts, amssymb, graphicx, color}
\usepackage[bookmarksnumbered, colorlinks, plainpages]{hyperref}

\newtheorem{theorem}{Theorem}[section]
\newtheorem{lemma}[theorem]{Lemma}

\newtheorem{corollary}[theorem]{Corollary}
\theoremstyle{definition}
\newtheorem{definition}[theorem]{Definition}
\newtheorem{example}[theorem]{Example}

\theoremstyle{remark}
\newtheorem{remark}[theorem]{Remark}
\numberwithin{equation}{section}

\begin{document}
\setcounter{page}{1}


\title [on higher $\{g_n, h_n\}$-derivations] {on higher $\{g_n, h_n\}$-derivations}
\author[Amin Hosseini and Nadeem Ur Rehman]{Amin Hosseini$^{\ast}$ and Nadeem Ur Rehman}
\address{ Amin Hosseini, Department of Mathematics, Kashmar Higher Education Institute- Kashmar- Iran}
\email{\textcolor[rgb]{0.00,0.00,0.84}{hosseini.amin82@gmail.com, a.hosseini@kashmar.ac.ir}}

\address{ Nadeem Ur Rehman, Department of Mathematics,
Aligarh Muslim University,
Aligarh-202002 India}
\email{\textcolor[rgb]{0.00,0.00,0.84}{rehman100@gmail.com}}

\subjclass[2010]{47B47; 16W10}

\keywords{derivation; $\{g, h\}$-derivation; higher $\{g_n, h_n\}$-derivation; Jordan higher $\{g_n, h_n\}$-derivation.}

\date{Received: xxxxxx; Revised: yyyyyy; Accepted: zzzzzz.
\newline \indent $^{*}$ Corresponding author}

\begin{abstract}
In this article, we introduce the concepts of higher $\{g_n, h_n\}$-derivation and Jordan higher $\{g_n, h_n\}$-derivation, and then we give a characterization of higher $\{g_n, h_n\}$-derivations in terms of $\{g, h\}$-derivations. Using this result, we prove that every Jordan higher $\{g_n, h_n\}$-derivation on a semiprime algebra is a higher $\{g_n, h_n\}$-derivation.
\end{abstract} \maketitle

\section{Introduction and preliminaries}
Recently, in 2016 Bre$\check{s}$ar \cite{B} introduced the notion of $\{g, h\}$-derivation. Let $\mathcal{A}$ be an algebra over a field $\mathbb{F}$ with $char(\mathbb{F}) \neq 2$, and let $f, g, h :\mathcal{A} \rightarrow \mathcal{A}$ be linear maps. We say that $f$ is a $\{g, h\}$-derivation if $f(ab) = g(a) b + a h(b) = h(a) b + a g(b)$ for all $a, b \in \mathcal{A}$, and $f$ is called a Jordan $\{g, h\}$-derivation if $f(a \circ b) = g(a) \circ b + a \circ h(b)$ for all $a, b \in \mathcal{A}$, where $a \circ b = a b + ba$. We call $a \circ b$ the Jordan product
of $a$ and $b$. It is evident that $a \circ b = b \circ a$ for all $a, b \in \mathcal{A}$. The notion of a Jordan $\{g, h\}$-derivation is a generalization of what is called a Jordan generalized derivation in \cite{L}. Recall that a linear mapping $f :\mathcal{A} \rightarrow \mathcal{A}$ is called a \emph{Jordan generalized derivation} if there exists a linear mapping
$d :\mathcal{A} \rightarrow \mathcal{A}$ such that $f(a \circ b) = f(a) \circ b + a \circ d(b)$ for all $a, b \in \mathcal{A},$ where $d$ is called an associated linear map of $f$. It is clear that $f(a \circ b) = d(a) \circ b + a \circ f(b)$ for all $a, b \in \mathcal{A}.$ Obviously, the definition of a Jordan generalized derivation is generally not equivalent to the ordinary Jordan case of generalized derivations.
For more details in this regard, see \cite{B, L} and references therein.

As an important result, Bre$\check{s}$ar \cite[Theorem 4.3]{B} established that every Jordan $\{g, h\}$-derivation of a semiprime algebra $\mathcal{A}$ is a $\{g, h\}$-derivation. He also showed that every Jordan $\{g, h\}$-derivation of the tensor product of a semiprime and a commutative algebra is a $\{g, h\}$-derivation. Obviously, every $\{g, h\}$-derivation is a Jordan $\{g, h\}$-derivation, but the converse is in general not true. For instance in this regard see \cite[Example 2.1]{B}.

In this study, we introduce the concept of a higher $\{g_n, h_n\}$-derivation, and then we characterize it on algebras. Throughout this paper, $\mathcal{A}$ denotes an algebra over a field of characteristic zero, and $I$ denotes the identity mapping on $\mathcal{A}$. Let $f$ be a $\{g, h\}$-derivation on an algebra $\mathcal{A}$. An easy induction argument implies that $f^{n}(ab) = \sum_{k = 0}^{n}\Big(_{k}^{n}\Big) g^{n - k}(a) h^{k}(b) = \sum_{k = 0}^{n}\Big(_{k}^{n}\Big) h^{n - k}(a) g^{k}(b)$ (Leibniz rule) for each $a, b \in \mathcal{A}$ and each non-negative integer $n$, where $f^{0} = g^{0} = h^{0} = I$. If we define the sequences $\{f_n\}$, $\{g_n\}$ and $\{h_n\}$ of linear mappings on $\mathcal{A}$ by $f_{0} = g_0 = h_0 = I$, and $f_n = \frac{f^n}{n!}$, $g_n = \frac{g^n}{n!}$ and $h_n = \frac{h^n}{n!}$, then it follows from the Leibniz rule that $f_n$'s, $g_n$'s and $h_n$'s satisfy
\begin{align}
f_n(ab) = \sum_{k = 0}^{n}g_{n - k}(a) h_k(b) = \sum_{k = 0}^{n}h_{n - k}(a) g_k(b),
\end{align}
for each $a, b \in \mathcal{A}$ and each non-negative integer $n$. This is our motivation to consider the sequences $\{f_n\}$, $\{g_n\}$ and $\{h_n\}$ of linear mappings on an algebra $\mathcal{A}$ satisfying (1.1). A sequence $\{f_n\}$ of linear mappings on $\mathcal{A}$ is called a higher $\{g_n, h_n\}$-derivation if there exist two sequences $\{g_n\}$ and $\{h_n\}$ of linear mappings on $\mathcal{A}$ satisfying $f_n(ab) = \sum_{k = 0}^{n}g_{n - k}(a) h_k(b) = \sum_{k = 0}^{n}h_{n - k}(a) g_k(b)$ for any $a, b \in \mathcal{A}$ and any non-negative integer $n$. Additionally, a sequence $\{f_n\}$ of linear mappings on $\mathcal{A}$ is called a Jordan higher $\{g_n, h_n\}$-derivation if
\begin{align}
f_n(a \circ b) = \sum_{k = 0}^{n}g_{n - k}(a) \circ h_k(b),
\end{align}
for each $a, b \in \mathcal{A}$ and each non-negative integer $n$.
Notice that if $\{f_n\}$ is a higher $\{f_n, f_n\}$-derivation (resp. Jordan higher $\{f_n, f_n\}$-derivation), then it is an ordinary higher derivation (resp. Jordan higher derivation). We know that if $f$ is a $\{g, h\}$-derivation, then $\{f_n = \frac{f^n}{n !}\}$ is a higher $\{g_n, h_n\}$-derivation, where $g_n = \frac{g^n}{n !}$, $h_n = \frac{h^n}{n !}$ and $f_0 = g_0 = h_0 = I$. We call this kind of higher $\{g_n, h_n\}$-derivation an ordinary higher $\{g_n, h_n\}$-derivation, but this is not the only example of a higher $\{g_n, h_n\}$-derivation.
\\
In 2010, Miravaziri \cite{M} characterized all higher derivations on an algebra $\mathcal{A}$ in terms of the derivations on $\mathcal{A}$. In this article, by getting idea from \cite{M}, our aim is to characterize all higher $\{g_n, h_n\}$-derivations on an algebra $\mathcal{A}$ in terms of the $\{g, h\}$-derivations on $\mathcal{A}$. Indeed, we show that each higher $\{g_n, h_n\}$-derivation is a combination of compositions of $\{g, h\}$-derivations. As the main result of this article, we prove that if $\{f_n\}$ is a higher $\{g_n, h_n\}$-derivation on an algebra $\mathcal{A}$ with $f_0 = g_0 = h_0 = I$, then there is a sequence $\{F_n\}$ of $\{G_n, H_n\}$-derivations on $\mathcal{A}$ such that
\[
\left\lbrace
  \begin{array}{c l}
     f_n = \sum_{i = 1}^{n}\Bigg(\sum_{\sum_{j = 1}^{i}r_j =
n}\Big(\prod_{j = 1}^{i}\frac{1}{r_j + ... +
r_i}\Big)F_{r_1}...F_{r_i}\Bigg), & \\ \\
 g_n = \sum_{i = 1}^{n}\Bigg(\sum_{\sum_{j = 1}^{i}r_j =
n}\Big(\prod_{j = 1}^{i}\frac{1}{r_j + ... +
r_i}\Big)G_{r_1}...G_{r_i}\Bigg), & \\ \\
 h_n = \sum_{i = 1}^{n}\Bigg(\sum_{\sum_{j = 1}^{i}r_j =
n}\Big(\prod_{j = 1}^{i}\frac{1}{r_j + ... +
r_i}\Big)H_{r_1}...H_{r_i}\Bigg),
  \end{array}
\right. \]
where the inner summation is taken over all positive integers $r_j$with $\sum_{j = 1}^{i}r_j = n$. The importance of this result is to transfer the problems such as the characterization of Jordan higher $\{g_n, h_n\}$-derivations on semiprime algebras and automatic continuity of higher $\{g_n, h_n\}$-derivations into the same problems concerning $\{g, h\}$-derivations. Let $\mathcal{A}$ be an algebra, $D$ be the set of all higher derivations $\{d_n\}_{n = 0, 1, ...}$ on $\mathcal{A}$ with $d_0 = I$ and $\Delta$ be the set of all sequences $\{\delta_n\}_{n = 0, 1, ...}$ of derivations on $\mathcal{A}$ with $\delta_0 = 0$.
As an application of the main result of this article, we investigate Jordan higher $\{g_n, h_n\}$-derivations on algebras. It is a classical question in which algebras (and rings) a Jordan derivation is necessarily a derivation. Let us give a brief background in this issue. In 1957, Herstein \cite{He} achieved a result which asserts any Jordan derivation on a prime ring of characteristic different from two is a derivation. A brief proof of Herstein's result can be found in \cite{b3}. In 1975, Cusack \cite{C} generalized Herstein's result to 2-torsion free semiprime rings (see also \cite{b1} for an alternative proof). Moreover, Vukman \cite{V} investigated generalized Jordan derivations on semiprime rings and he proved that every \emph{generalized Jordan derivation} of a 2-torsion free semiprime ring is a generalized derivation. Recently, the first name author along with Ajda Fo$\check{s}$ner \cite{H} have studied the same problem for $(\sigma, \tau$)-derivations from a $C^{\ast}$-algebra $\mathcal{A}$ into a Banach $\mathcal{A}$-module $\mathcal{M}$.
In this paper, we show that if $\{f_n\}$ is a Jordan higher $\{g_n, h_n\}$-derivation of a semiprime algebra $\mathcal{A}$ with $f_0 = g_0 = h_0 = I$, then it is a higher $\{g_n, h_n\}$-derivation.

\section{characterization of higher $\{g_n, h_n\}$-derivations on algebras}
Throughout the article, $\mathcal{A}$ denotes an algebra over a field of characteristic zero, and $I$ is the identity mapping on $\mathcal{A}$. Let $f, g, h:\mathcal{A} \rightarrow \mathcal{A}$ be linear maps. We say that $f$ is a $\{g, h\}$-derivation if $f(ab) = g(a) b + a h(b) = h(a) b + a g(b)$ for all $a, b \in \mathcal{A}$, and $f$ is called a Jordan $\{g, h\}$-derivation if $f(a \circ b) = g(a) \circ b + a \circ h(b)$ for all $a, b \in \mathcal{A}$, where $a \circ b = ab + ba$. \\

We begin with the following definition.
\begin{definition}A sequence $\{f_n\}$ of linear mappings on $\mathcal{A}$ is called a higher $\{g_n, h_n\}$-derivation if there are two sequences $\{g_n\}$ and $\{h_n\}$ of linear mappings on $\mathcal{A}$ such that $f_n(ab) = \sum_{k = 0}^{n}g_{n - k}(a) h_k(b) = \sum_{k = 0}^{n}h_{n - k}(a) g_k(b)$ for each $a, b \in \mathcal{A}$ and each non-negative integer $n$.
\end{definition}

We begin our results with the following lemma which will be used extensively to prove the main theorem of this article. The following lemma has been motivated by \cite{M}.
\begin{lemma} \label{1} Let $\{f_n\}$ be a higher $\{g_n, h_n\}$-derivation on an algebra $\mathcal{A}$ with $f_0 = g_0 = h_0 =I$. Then there is a sequence $\{F_n\}$ of $\{G_n, H_n\}$-derivations on $\mathcal{A}$ such that
\[
\left\lbrace
  \begin{array}{c l}
     (n + 1) f_{n + 1} = \sum_{k = 0}^{n}F_{k + 1}f_{n - k}, & \\ \\
 (n + 1) g_{n + 1} = \sum_{k = 0}^{n}G_{k + 1}g_{n - k}, & \\ \\
 (n + 1) h_{n + 1} = \sum_{k = 0}^{n}H_{k + 1}h_{n - k}
  \end{array}
\right. \]
for each non-negative integer $n$.
\end{lemma}

\begin{proof} Using induction on $n$, we prove this lemma. Let $n = 0$. We know that $f_1(ab) = g_1(a) b + a h_1(b) = h_1(a) b + a g_1(b)$ for all $a, b \in \mathcal{A}$. Thus, if $F_1 = f_1$, $G_1 = g_1$ and $H_1 = h_1$, then $F_1$ is a $\{G_1, H_1\}$-derivation on $\mathcal{A}$ and further, $(0 + 1)f_{0 + 1} = \sum_{k = 0}^{0}F_{k + 1}f_{0 - k}$, $(0 + 1)g_{0 + 1} = \sum_{k = 0}^{0}G_{k + 1}g_{0 - k}$ and $(0 + 1)h_{0 + 1} = \sum_{k = 0}^{0}H_{k + 1}h_{0 - k}$. As induction assumption, suppose that $F_k$ is a $\{G_k, H_k\}$-derivation for any $k \leq n$ and further \[
\left\lbrace
  \begin{array}{c l}
     (r + 1) f_{r + 1} = \sum_{k = 0}^{r}F_{k + 1}f_{r - k}, & \\ \\
 (r + 1) g_{r + 1} = \sum_{k = 0}^{r}G_{k + 1}g_{r - k}, & \\ \\
 (r + 1) h_{r + 1} = \sum_{k = 0}^{r}H_{k + 1}h_{r - k}
  \end{array}
\right. \]   for $ r = 0, 1, ..., n - 1$. Put $F_{n + 1} = (n + 1) f_{n + 1} - \sum_{k = 0}^{n - 1}F_{k + 1}f_{n - k}$, $G_{n + 1} = (n + 1)g_{n + 1} - \sum_{k = 0}^{n - 1}G_{k + 1}g_{n - k}$ and $H_{n + 1} = (n + 1)h_{n + 1} - \sum_{k = 0}^{n - 1}H_{k + 1}h_{n - k}$. Our next task is to show that $F_{n + 1}$ is a $\{G_{n + 1}, H_{n + 1}\}$-derivation on $\mathcal{A}$. For $a, b \in \mathcal{A}$, we have
\begin{align*}
F_{n + 1}(ab) & = (n + 1) f_{n + 1}(ab) - \sum_{k = 0}^{n - 1}F_{k + 1}f_{n - k}(ab) \\ & = (n + 1) \sum_{k = 0}^{n + 1}g_k(a) h_{n + 1 - k}(b) - \sum_{k = 0}^{n - 1}F_{k + 1} \Big(\sum_{l = 0}^{n - k}g_l(a) h_{n - k - l}(b)\Big).
\end{align*}
So, we have
\begin{align*}
F_{n + 1}(ab) & = \sum_{k = 0}^{n + 1}(n + 1) g_k(a) h_{n + 1 - k}(b) - \sum_{k = 0}^{n - 1}F_{k + 1} \Big(\sum_{l = 0}^{n - k}g_l(a) h_{n - k - l}(b)\Big) \\ & = \sum_{k = 0}^{n + 1}(k + n + 1 - k)g_k(a) h_{n + 1 - k}(b) - \sum_{k = 0}^{n - 1}F_{k + 1} \Big(\sum_{l = 0}^{n - k}g_l(a) h_{n - k - l}(b)\Big).
\end{align*}
Since $F_k$ is a $\{G_k, H_k\}$-derivation for each $k = 1, 2, ..., n$,
\begin{align*}
F_{n + 1}(ab) & = \sum_{k = 0}^{n + 1}k g_{k}(a) h_{n + 1 - k}(b) + \sum_{k = 0}^{n + 1}g_k(a) (n + 1 - k) h_{n + 1 - k}(b) \\ & - \sum_{k = 0}^{n - 1}\sum_{l = 0}^{n - k}\Big[G_{k + 1}\big(g_l(a)\big)h_{n - k - l}(b) + g_l(a) H_{k + 1}\big(h_{n - k - l}(b)\big) \Big].
\end{align*}
Letting
\begin{align*}
& G = \sum_{k = 0}^{n + 1}k g_{k}(a) h_{n + 1 - k}(b) - \sum_{k = 0}^{n - 1}\sum_{l = 0}^{n - k}G_{k + 1}\big(g_l(a)\big)h_{n - k - l}(b), \\ \\ & H = \sum_{k = 0}^{n + 1} g_{k}(a)(n + 1 - k) h_{n + 1 - k}(b) - \sum_{k = 0}^{n - 1}\sum_{l = 0}^{n - k}g_l(a)H_{k + l}\big(h_{n - k - l}(b)\big),
\end{align*}
we have $F_{n + 1}(ab) = G + H$. Here, we compute $G$ and $H$. In the summation $\sum_{k = 0}^{n - 1}\sum_{l = 0}^{n - k}$, we have $0 \leq k + l \leq n$ and $k \neq n$. Thus if we put $r = k + l$, then we can write it as the form $\sum_{r = 0}^{n} \sum_{k + l = r, k \neq n}$. Putting $l = r - k$, we find that
\begin{align*}
G & = \sum_{k = 0}^{n + 1}k g_{k}(a) h_{n + 1 - k}(b) - \sum_{r = 0}^{n}\sum_{0 \leq k \leq r, k \neq n}G_{k + 1}\big(g_{r - k}(a)\big)h_{n - r}(b) \\ & = \sum_{k = 0}^{n + 1}k g_{k}(a) h_{n + 1 - k}(b) - \sum_{r = 0}^{n - 1}\sum_{k = 0}^{r}G_{k + 1}\big(g_{r - k}(a)\big)h_{n - r}(b) - \sum_{k = 0}^{n - 1}G_{k + 1}\big(g_{n - k}(a)\big)b.
\end{align*}
It means that
$$ G + \sum_{k = 0}^{n - 1}G_{k + 1}\big(g_{n - k}(a)\big)b = \sum_{k = 0}^{n + 1}k g_{k}(a) h_{n + 1 - k}(b) - \sum_{r = 0}^{n - 1}\sum_{k = 0}^{r}G_{k + 1}\big(g_{r - k}(a)\big)h_{n - r}(b).$$
Putting $r + 1$ instead of $k$ in the first summation of above, we have
\begin{align*}
G & + \sum_{k = 0}^{n - 1}G_{k + 1}\big(g_{n - k}(a)\big)b \\ & = \sum_{r = 0}^{n}(r + 1) g_{r + 1}(a) h_{n - r}(b) - \sum_{r = 0}^{n - 1}\sum_{k = 0}^{r}G_{k + 1}\big(g_{r - k}(a)\big)h_{n - r}(b) \\ & = \sum_{r = 0}^{n - 1}\Big[(r + 1)g_{r + 1}(a) - \sum_{k = 0}^{r}G_{k + 1}\big(g_{r - k}(a)\big)\Big]h_{n - r}(b) + (n + 1)g_{n + 1}(a)b.
\end{align*}
According to the induction hypothesis, $(r + 1) g_{r + 1}(a) = \sum_{k = 0}^{r}G_{k + 1}\big(g_{r - k}(a)\big)$ for $r = 0, ..., n - 1$. So, it is obtained that
\begin{align*}
G = \Big[(n + 1)g_{n + 1}(a) - \sum_{k = 0}^{n - 1}G_{k + 1}\big(g_{n - k}(a)\big)\Big]b = G_{n + 1}(a)b.
\end{align*}
Like above, we achieve that
\begin{align*}
H = a\Big[(n + 1)h_{n + 1} (b) - \sum_{k = 0}^{n - 1}H_{k + 1}\big(h_{n - k}(b)\big)\Big] = a H_{n + 1}(b).
\end{align*}
Therefore, we have $F_{n + 1}(ab) = G + H = G_{n + 1}(a)b + a H_{n + 1}(b)$. In the next step, we will show that $F_{n + 1}(ab) = H_{n + 1}(a) b + a G_{n + 1}(b)$. We have
\begin{align*}
F_{n + 1}(ab) & = (n + 1) f_{n + 1}(ab) - \sum_{k = 0}^{n - 1}F_{k + 1}f_{n - k}(ab) \\ & = (n + 1) \sum_{k = 0}^{n + 1}h_k(a) g_{n + 1 - k}(b) - \sum_{k = 0}^{n - 1}F_{k + 1} \Big(\sum_{l = 0}^{n - k}h_l(a) g_{n - k - l}(b)\Big).
\end{align*}
So,
\begin{align*}
F_{n + 1}(ab) & = \sum_{k = 0}^{n + 1}(n + 1) h_k(a) g_{n + 1 - k}(b) - \sum_{k = 0}^{n - 1}F_{k + 1} \Big(\sum_{l = 0}^{n - k}h_l(a) g_{n - k - l}(b)\Big) \\ & = \sum_{k = 0}^{n + 1}(k + n + 1 - k)h_k(a) g_{n + 1 - k}(b) - \sum_{k = 0}^{n - 1}F_{k + 1} \Big(\sum_{l = 0}^{n - k}h_l(a) g_{n - k - l}(b)\Big).
\end{align*}
Since $F_k$ is a $\{G_k, H_k\}$-derivation for each $k = 1, 2, ..., n$,
\begin{align*}
F_{n + 1}(ab) & = \sum_{k = 0}^{n + 1}k h_{k}(a) g_{n + 1 - k}(b) + \sum_{k = 0}^{n + 1}h_k(a) (n + 1 - k) g_{n + 1 - k}(b) \\ & - \sum_{k = 0}^{n - 1}\sum_{l = 0}^{n - k}\Big[H_{k + 1}\big(h_l(a)\big)g_{n - k - l}(b) + h_l(a) G_{k + 1}\big(g_{n - k - l}(b)\big) \Big].
\end{align*}
Letting
\begin{align*}
& \mathfrak{G} = \sum_{k = 0}^{n + 1}k h_{k}(a) g_{n + 1 - k}(b) - \sum_{k = 0}^{n - 1}\sum_{l = 0}^{n - k}H_{k + 1}\big(h_l(a)\big)g_{n - k - l}(b), \\ \\ & \mathfrak{H} = \sum_{k = 0}^{n + 1} h_{k}(a)(n + 1 - k) g_{n + 1 - k}(b) - \sum_{k = 0}^{n - 1}\sum_{l = 0}^{n - k}h_l(a)G_{k + l}\big(g_{n - k - l}(b)\big),
\end{align*}
we have $F_{n + 1}(ab) = \mathfrak{G} + \mathfrak{H}$. The next step is to compute $\mathfrak{G}$ and $\mathfrak{H}$. In the summation $\sum_{k = 0}^{n - 1}\sum_{l = 0}^{n - k}$, we have $0 \leq k + l \leq n$ and $k \neq n$. Thus if we put $r = k + l$, then we can write it as the form $\sum_{r = 0}^{n} \sum_{k + l = r, k \neq n}$. Putting $l = r - k$, we find that
\begin{align*}
\mathfrak{G} & = \sum_{k = 0}^{n + 1}k h_{k}(a) g_{n + 1 - k}(b) - \sum_{r = 0}^{n}\sum_{0 \leq k \leq r, k \neq n}H_{k + 1}\big(h_{r - k}(a)\big)g_{n - r}(b) \\ & = \sum_{k = 0}^{n + 1}k h_{k}(a) g_{n + 1 - k}(b) - \sum_{r = 0}^{n - 1}\sum_{k = 0}^{r}H_{k + 1}\big(h_{r - k}(a)\big)g_{n - r}(b) - \sum_{k = 0}^{n - 1}H_{k + 1}\big(h_{n - k}(a)\big)b.
\end{align*}
It means that
$$ \mathfrak{G} + \sum_{k = 0}^{n - 1}H_{k + 1}\big(h_{n - k}(a)\big)b = \sum_{k = 0}^{n + 1}k h_{k}(a) g_{n + 1 - k}(b) - \sum_{r = 0}^{n - 1}\sum_{k = 0}^{r}H_{k + 1}\big(h_{r - k}(a)\big)g_{n - r}(b).$$
Putting $r + 1$ instead of $k$ in the first summation of above, we have
\begin{align*}
\mathfrak{G} & + \sum_{k = 0}^{n - 1}H_{k + 1}\big(h_{n - k}(a)\big)b \\ & = \sum_{r = 0}^{n}(r + 1) h_{r + 1}(a) g_{n - r}(b) - \sum_{r = 0}^{n - 1}\sum_{k = 0}^{r}H_{k + 1}\big(h_{r - k}(a)\big)g_{n - r}(b) \\ & = \sum_{r = 0}^{n - 1}\Big[(r + 1)h_{r + 1}(a) - \sum_{k = 0}^{r}H_{k + 1}\big(h_{r - k}(a)\big)\Big]g_{n - r}(b) + (n + 1)h_{n + 1}(a)b.
\end{align*}
According to the induction hypothesis, $(r + 1) h_{r + 1}(a) = \sum_{k = 0}^{r}H_{k + 1}\big(h_{r - k}(a)\big)$ for $r = 0, ..., n - 1$. So, it is obtained that
\begin{align*}
\mathfrak{G} = \Big[(n + 1)h_{n + 1}(a) - \sum_{k = 0}^{n - 1}H_{k + 1}\big(h_{n - k}(a)\big)\Big]b = H_{n + 1}(a)b.
\end{align*}
By a reasoning like above, we get that
\begin{align*}
\mathfrak{H} = a\Big[(n + 1)g_{n + 1} (b) - \sum_{k = 0}^{n - 1}G_{k + 1}\big(g_{n - k}(b)\big)\Big] = a G_{n + 1}(b).
\end{align*}
Therefore, we have $F_{n + 1}(ab) = \mathfrak{G} + \mathfrak{H} = H_{n + 1}(a)b + a G_{n + 1}(b)$. Consequently, $F_{n + 1}$ is a $\{G_{n + 1}, H_{n + 1}\}$-derivation and the proof is complete.
\end{proof}

\begin{example} \label{*} Using Lemma \ref{1}, the first five terms of a higher $\{g_n, h_n\}$-derivation $\{f_n\}$ are
\begin{align*}
&f_0 = I, \\ &
f_1  = F_1,\\ &
2f_2 = F_1 f_1 + F_2 f_0 = F_1 F_1 + F_2,\\ &
f_2 = \frac{1}{2}F_1^{2} + \frac{1}{2}F_2,\\ &
3f_3 = F_1 f_2 + F_2 f_1 + F_3 f_0 = F_1(\frac{1}{2}F_1^{2} + \frac{1}{2}F_2) + F_2 F_1 + F_3, \\ &
f_3 = \frac{1}{6}F_1^{3} + \frac{1}{6}F_1 F_2 + \frac{1}{3} F_2 F_1 + \frac{1}{3}F_3, \\ &
4f_4 = F_1 f_3 + F_2 f_2 + F_3 f_1 + F_4 f_0 \\ & = F_1 \Big(\frac{1}{6}F_1^{3} + \frac{1}{6}F_1 F_2 + \frac{1}{3} F_2 F_1 + \frac{1}{3}F_3\Big) + F_2 \Big( \frac{1}{2}F_1^{2} + \frac{1}{2}F_2\Big) + F_3 F_1 + F_4, \\ &
f_4 = \frac{1}{24}F_1^{4} + \frac{1}{24}F_1^2 F_2 + \frac{1}{12} F_1 F_2 F_1 + \frac{1}{12}F_1 F_3 + \frac{1}{8}F_2 F_1^2 + \frac{1}{8}F_{2}^2 + \frac{1}{4}F_3 F_1 + \frac{1}{4}F_4.
\end{align*}
\end{example}

Now the main result of this paper reads as follows:
\begin{theorem}\label{2}Let $\{f_n\}$ be a higher $\{g_n, h_n\}$-derivation on an algebra $\mathcal{A}$ with $f_0 = g_0 = h_0 = I$. Then there is a sequence $\{F_n\}$ of $\{G_n, H_n\}$-derivations on $\mathcal{A}$ such that

\[
\left\lbrace
  \begin{array}{c l}
     f_n = \sum_{i = 1}^{n}\Bigg(\sum_{\sum_{j = 1}^{i}r_j =
n}\Big(\prod_{j = 1}^{i}\frac{1}{r_j + ... +
r_i}\Big)F_{r_1}...F_{r_i}\Bigg), & \\ \\
 g_n = \sum_{i = 1}^{n}\Bigg(\sum_{\sum_{j = 1}^{i}r_j =
n}\Big(\prod_{j = 1}^{i}\frac{1}{r_j + ... +
r_i}\Big)G_{r_1}...G_{r_i}\Bigg), & \\ \\
 h_n = \sum_{i = 1}^{n}\Bigg(\sum_{\sum_{j = 1}^{i}r_j =
n}\Big(\prod_{j = 1}^{i}\frac{1}{r_j + ... +
r_i}\Big)H_{r_1}...H_{r_i}\Bigg),
  \end{array}
\right. \]
where the inner summation is taken over all positive integers $r_j$with $\sum_{j = 1}^{i}r_j = n$.
\end{theorem}
\begin{proof}First, we show that if $f_n$, $g_n$ and $h_n$ are of the above form, then they satisfy the recursive relations of Lemma \ref{1}. Simplifying the notation, we put $a_{r_1, ..., r_i} = \prod_{j = 1}^{i}\frac{1}{r_j + ... + r_i}$. Note that if $r_1 + ... + r_i = n + 1$, then $(n + 1) a_{r_1, ..., r_i} = a_{r_2, ..., r_i}$. Furthermore, $a_{n + 1} = \frac{1}{n + 1}$. According to the aforementioned assumptions, we have
\begin{align*}
f_{n + 1 } & = \sum_{i = 2}^{n + 1}\Bigg(\sum_{\sum_{j = 1}^{i}r_j = n + 1}a_{r_1, ... r_i}F_{r_1}...F_{r_i}\Bigg) + a_{n + 1}F_{n + 1} \\ & = \sum_{i = 2}^{n + 1}\Bigg(\sum_{\sum_{j = 1}^{i}r_j = n + 1}a_{r_1, ... r_i}F_{r_1}...F_{r_i}\Bigg) + \frac{F_{n + 1}}{n + 1}.
\end{align*}
So,
\begin{align*}
(n + 1)f_{n + 1 } & = \sum_{i = 2}^{n + 1}\Bigg(\sum_{\sum_{j = 1}^{i}r_j = n + 1}(n + 1) a_{r_1, ... r_i}F_{r_1}...F_{r_i}\Bigg) + F_{n + 1} \\ & = \sum_{i = 2}^{n + 1}\Bigg(\sum_{\sum_{j = 1}^{i}r_j = n + 1} a_{r_2, ... r_i}F_{r_1}...F_{r_i}\Bigg) + F_{n + 1} \\ & = \sum_{i = 2}^{n + 1}\Bigg(\sum_{r_1 = 1}^{n + 2 - i} F_{r_1} \sum_{\sum_{j = 2}^{i}r_j = n + 1 - r_1}a_{r_2, ... r_i}F_{r_2}...F_{r_i}\Bigg) + F_{n +1} \\ & = \sum_{r_1 = 1}^{n}F_{r_1} \sum_{i = 2}^{n - (r_1 - 1)} \Bigg (\sum_{\sum_{j = 2}^{i}r_j = n - (r_1 - 1)}a_{r_2, ... r_i}F_{r_2}...F_{r_i} \Bigg) + F_{n + 1} \\ & = \sum_{r_1 = 1}^{n}F_{r_1} f_{n - (r_1 - 1)} + F_{n + 1} \\ & = \sum_{k = 0}^{n}F_{k + 1}f_{n - k}.
\end{align*}
Using a reasoning like above, we get that
\[
\left\lbrace
  \begin{array}{c l}
 (n + 1) g_{n + 1} = \sum_{k = 0}^{n}G_{k + 1}g_{n - k}, & \\ \\
 (n + 1) h_{n + 1} = \sum_{k = 0}^{n}H_{k + 1}h_{n - k},
  \end{array}
\right. \]
for each non-negative integer $n$. Putting $n + 1 = m$, we find that $mf_{m} = \sum_{k = 0}^{m - 1}F_{k + 1}f_{m - 1 - k} = \sum_{k = 0}^{m - 2}F_{k + 1}f_{m - 1 - k} + F_m$, and consequently
$$F_m = mf_m - \sum_{k = 0}^{m - 2}F_{k + 1}f_{m - 1 - k}.$$ Similarly, we have
\[
\left\lbrace
  \begin{array}{c l}
 G_m = mg_m - \sum_{k = 0}^{m - 2}G_{k + 1}g_{m - 1 - k}, & \\ \\
 H_m = mh_m - \sum_{k = 0}^{m - 2}H_{k + 1}h_{m - 1 - k}.
  \end{array}
\right. \]

Therefore, we can define $F_n, G_n, H_n : \mathcal{A} \rightarrow \mathcal{A}$ by $F_0 = G_0 = H_0 = 0$ and
\[
\left\lbrace
  \begin{array}{c l}
     F_n = n f_n - \sum_{k = 0}^{n - 2}F_{k + 1}f_{n - 1 - k}, & \\ \\
 G_n = n g_n - \sum_{k = 0}^{n- 2}G_{k + 1}g_{n - 1 - k}, & \\ \\
 H_n = n h_n - \sum_{k = 0}^{n - 2}H_{k + 1}h_{n - 1 - k},
  \end{array}
\right. \]
for each positive integer $n$. It follows from Lemma \ref{1} that $\{F_n\}$ is a sequence of $\{G_n, H_n\}$-derivations. In addition, we prove that if $f_n$, $g_n$ and $h_n$ are of the form
\[
\left\lbrace
  \begin{array}{c l}
 (n + 1) f_{n + 1} = \sum_{k = 0}^{n}F_{k + 1}f_{n - k}, & \\ \\
 (n + 1) g_{n + 1} = \sum_{k = 0}^{n}G_{k + 1}g_{n - k}, & \\ \\
 (n + 1) h_{n + 1} = \sum_{k = 0}^{n}H_{k + 1}h_{n - k},
  \end{array}
\right. \]
where $\{F_n\}$ is a sequence of $\{G_n, H_n\}$-derivations, then $\{f_n\}$ is a higher $\{g_n, h_n\}$-derivation on $\mathcal{A}$ with $f_0 = g_0 = h_0 = I$. To see this, we use induction on $n$. For $n = 0$, we have $f_0(ab) = ab = g_0(a)h_0(b) = h_0(a) g_0(b)$. As the inductive hypothesis, assume that
\begin{align*}
f_k(ab) = \sum_{i = 0}^{k}g_{i}(a)h_{k - i}(b) = \sum_{i = 0}^{k}h_{i}(a)g_{k - i}(b) \ \ for \ k \leq n.
\end{align*}

Therefore, we have
\begin{align*}
(n + 1)f_{n + 1}(ab) & = \sum_{k = 0}^{n}F_{k + 1}f_{n - k}(ab) \\ & = \sum_{k = 0}^{n}F_{k + 1}\sum_{i = 0}^{n - k}g_i(a) h_{n - k - i}(b) \\ & = \sum_{i = 0}^{n}\Bigg(\sum_{k = 0}^{n - i}G_{k + 1}g_{n - k - i}(a)\Bigg)h_i(b) + \sum_{i = 0}^{n}g_i(a)\Bigg(\sum_{k = 0}^{n - i}H_{k + 1}h_{n - k - i}(b)\Bigg).
\end{align*}
According to the above-mentioned recursive relations, we continue the previous expressions as follows:
\begin{align*}
(n + 1)f_{n + 1}(ab) & = \sum_{i = 0}^{n}(n - i + 1)g_{n - i + 1}(a)h_i(b) + \sum_{i = 0}^{n}g_i(a)(n - i + 1)h_{n - i + 1}(b) \\ & = \sum_{i = 1}^{n + 1}ig_{i}(a)h_{n + 1 - i}(b) + \sum_{i = 0}^{n}(n - i + 1)g_i(a)h_{n + 1 - i}(b) \\ & = (n + 1)\sum_{i = 0}^{n + 1}g_{i}(a)h_{n + 1 - i}(b),
\end{align*}
which means that $f_{n + 1}(ab) = \sum_{i = 0}^{n + 1}g_{i}(a)h_{n + 1 - i}(b)$. By an argument like above, we can obtain that $f_{n + 1}(ab) = \sum_{i = 0}^{n + 1}h_{i}(a)g_{n + 1 - i}(b)$. Thus, $\{f_n\}$ is a higher $\{g_n, h_n\}$-derivation on $\mathcal{A}$ which is characterized by the sequence $\{F_n\}$ of $\{G_n, H_n\}$-derivations. This completes the proof.
\end{proof}
 In the next example, using the above theorem, we characterize term $f_4$ of a higher $\{g_n, h_n\}$-derivation $\{f_n\}$.
\begin{example} We compute the coefficients $a_{r_1, ..., r_i}$ for the case $n = 4$. First, note that
$4 = 1 + 3 = 3 + 1 = 2 + 2 = 1 + 1 + 2 = 1 + 2 + 1 = 2 + 1 + 1 = 1 + 1 + 1 + 1.$
Based on the definition of $a_{r_1, ..., r_i}$ we have
\begin{align*}
& a_4 = \frac{1}{4}, \\ &
a_{1, 3} = \frac{1}{1 + 3}. \frac{1}{3} = \frac{1}{12}, \\ &
a_{3, 1} = \frac{1}{3 + 1}. \frac{1}{1} = \frac{1}{4}, \\ &
a_{2, 2} = \frac{1}{2 + 2}. \frac{1}{2} = \frac{1}{8}, \\ &
a_{1, 1, 2} = \frac{1}{1 + 1 + 2}. \frac{1}{1 + 2}. \frac{1}{2} = \frac{1}{24}, \\ &
a_{1, 2, 1} = \frac{1}{1 + 2 + 1}. \frac{1}{2 + 1}. \frac{1}{1} = \frac{1}{12}, \\ &
a_{2, 1, 1} = \frac{1}{2 + 1 + 1}. \frac{1}{1 + 1}. \frac{1}{1} = \frac{1}{8}, \\ &
a_{1, 1,1, 1} = \frac{1}{1 + 1 + 1 + 1}. \frac{1}{1 + 1 + 1}. \frac{1}{1 + 1}. \frac{1}{1} = \frac{1}{24}.
\end{align*}
Therefore, $f_4$ is characterized as follows:\\
\begin{align*}
f_ 4 = & \frac{1}{4}F_4 + \frac{1}{12}F_1 F_3 + \frac{1}{4}F_3 F_1 + \frac{1}{8}F_{2} F_2 + \frac{1}{24}F_1 F_1 F_2 + \frac{1}{12} F_1 F_2 F_1 \\ & + \frac{1}{8}F_2 F_1 F_1 + \frac{1}{24}F_1 F_1 F_1 F_1.
\end{align*}
\end{example}

In the following there are some immediate consequences of the above theorem. Before it, recall that a sequence $\{f_n\}$ of linear mappings on $\mathcal{A}$ is called a Jordan higher $\{g_n, h_n\}$-derivation if there exist two sequences $\{g_n\}$ and $\{h_n\}$ of linear mappings on $\mathcal{A}$ such that $f_n(a \circ b) = \sum_{k = 0}^{n}g_{n - k}(a) \circ h_k(b)$ holds for each $a, b \in \mathcal{A}$ and each non-negative integer $n$.
Since the Jordan product is commutative, we have
\begin{align*}
f_n(a \circ b) = f_n(b \circ a) & = \sum_{k = 0}^{n}g_{n - k}(b) \circ h_k(a) \\ & = \sum_{k = 0}^{n}g_{k}(b) \circ h_{n - k}(a) \\ & = \sum_{k = 0}^{n}h_{n - k}(a) \circ g_k(b).
\end{align*}
So, it is observed that if $\{f_n\}$ is a Jordan higher $\{g_n, h_n\}$-derivation, then
$$f_n(a \circ b) = \sum_{k = 0}^{n}g_{n - k}(a) \circ h_k(b) = \sum_{k = 0}^{n}h_{n - k}(a) \circ g_k(b),$$ for all $a, b \in \mathcal{A}$.

\begin{corollary} \label{6} Let $\{f_n\}$ be a Jordan higher $\{g_n, h_n\}$-derivation on a semiprime algebra $\mathcal{A}$ with $f_0 = g_0 = h_0 = I$. Then $\{f_n\}$ is a higher $\{g_n, h_n\}$-derivation.
\end{corollary}

\begin{proof} Using the proof of Theorem \ref{2}, we can show that if $\{f_n\}$ is a Jordan higher $\{g_n, h_n\}$-derivation on an algebra $\mathcal{A}$ with $f_0 = g_0 = h_0 = I$, then there exists a sequence $\{F_n\}$ of Jordan $\{G_n, H_n\}$-derivations on $\mathcal{A}$ such that
\begin{align*}
f_n = \sum_{i = 1}^{n}\big(\sum_{\sum_{j = 1}^{i}r_j =
n}\big(\prod_{j = 1}^{i}\frac{1}{r_j + ... +
r_i}\big)F_{r_1}...F_{r_i}\big),
\end{align*}
where the inner summation is taken over all positive integers $r_j$ with $\sum_{j = 1}^{i}r_j = n$. Since $\mathcal{A}$ is a semiprime algebra, \cite[Theorem 4.3]{B} proves the corollary.
\end{proof}

\begin{remark}We know that the notion of a Jordan $\{g, h\}$-derivation is a generalization of a Jordan generalized derivation (see Introduction). A sequence $\{f_n\}$ of linear mappings on an algebra $\mathcal{A}$ is called a \emph{Jordan higher generalized derivation} if there exists a sequence $\{d_n\}$ of linear mappings on $\mathcal{A}$ such that $f_n(a \circ b) = \sum_{k = 0}^{n}f_{n - k}(a) \circ d_k(b)$ for all $a, b \in \mathcal{A}.$  Obviously, if $\{f_n\}$ is a Jordan higher generalized derivation associated with a sequence $\{d_n\}$ of linear mappings on $\mathcal{A}$, then
it is a Jordan higher $\{f_n, d_n\}$-derivation. So, Corollary \ref{6} is also valid for \emph{Jordan higher generalized derivations}.
\end{remark}

\begin{theorem} If $\{f_n\}_{n = 0, 1, ...}$ is a higher $\{g_n, h_n\}$-derivation on $\mathcal{A}$ with $f_0 = g_0 = h_0 = I$, then there is a sequence $\{F_n\}_{n = 0, 1, ...}$ of $\{G_n, H_n\}$-derivations with $F_0 = G_0 = H_0 = 0$ characterizing the higher $\{g_n, h_n\}$-derivation $\{f_n\}_{n = 0, 1, ...}$. As well as if $\{F_n\}_{n = 0, 1, ...}$ is a sequence of $\{G_n, H_n\}$-derivations with $F_0 = G_0 = H_0 = 0$, then there exists a higher $\{g_n, h_n\}$-derivation $\{f_n\}$ with $f_0 = g_0 = h_0 = I$ which is characterized by the sequence $\{F_n\}_{n = 0, 1, ...}$.
\end{theorem}

\begin{proof}Let $\{f_n\}$ be a higher $\{g_n, h_n\}$-derivation on $\mathcal{A}$ with $f_0 = g_0 = h_0 = I$. We are going to obtain a sequence $\{F_n\}_{n = 0, 1, ...}$ of $\{G_n, H_n\}$-derivations with $F_0 = G_0 = H_0 = 0$ that characterizes the higher $\{g_n, h_n\}$-derivation $\{f_n\}$. Define $F_n, G_n, H_n : \mathcal{A} \rightarrow \mathcal{A}$ by $F_0 = G_0 = H_0 = 0$ and
\[
\left\lbrace
  \begin{array}{c l}
     F_n = n f_n - \sum_{k = 0}^{n - 2}F_{k + 1}f_{n - 1 - k}, & \\ \\
 G_n = n g_n - \sum_{k = 0}^{n- 2}G_{k + 1}g_{n - 1 - k}, & \\ \\
 H_n = n h_n - \sum_{k = 0}^{n - 2}H_{k + 1}h_{n - 1 - k},
  \end{array}
\right. \]
for each positive integer $n$. Then it follows from Lemma \ref{1} that $\{F_n\}$ is a sequence of $\{G_n, H_n\}$-derivations characterizing the higher $\{g_n, h_n\}$-derivation $\{f_n\}$. Now suppose that $\{F_n\}_{n = 0, 1, ...}$ is a sequence of $\{G_n, H_n\}$-derivations with $F_0 = G_0 = H_0 = 0$. We will show that there exists a higher $\{g_n, h_n\}$-derivation $\{f_n\}$ with $f_0 = g_0 = h_0 = I$ which is characterized by the sequence $\{F_n\}_{n = 0, 1, ...}$. We define $f_n, g_n, h_n : \mathcal{A} \rightarrow \mathcal{A}$ by $f_0 = g_0 = h_0 = I$ and
\[
\left\lbrace
  \begin{array}{c l}
     f_n = \sum_{i = 1}^{n}\Bigg(\sum_{\sum_{j = 1}^{i}r_j =
n}\Big(\prod_{j = 1}^{i}\frac{1}{r_j + ... +
r_i}\Big)F_{r_1}...F_{r_i}\Bigg), & \\ \\
 g_n = \sum_{i = 1}^{n}\Bigg(\sum_{\sum_{j = 1}^{i}r_j =
n}\Big(\prod_{j = 1}^{i}\frac{1}{r_j + ... +
r_i}\Big)G_{r_1}...G_{r_i}\Bigg), & \\ \\
 h_n = \sum_{i = 1}^{n}\Bigg(\sum_{\sum_{j = 1}^{i}r_j =
n}\Big(\prod_{j = 1}^{i}\frac{1}{r_j + ... +
r_i}\Big)H_{r_1}...H_{r_i}\Bigg).
  \end{array}
\right. \]
By Theorem \ref{2}, $f_n$, $g_n$ and $h_n$ satisfy the following recursive relations:
\[
\left\lbrace
  \begin{array}{c l}
(n + 1) f_{n + 1} = \sum_{k = 0}^{n}F_{k + 1}f_{n - k}, & \\ \\
(n + 1) g_{n + 1} = \sum_{k = 0}^{n}G_{k + 1}g_{n - k}, & \\ \\
(n + 1) h_{n + 1} = \sum_{k = 0}^{n}H_{k + 1}h_{n - k}.
  \end{array}
\right. \]

Based on the last part of the proof of Theorem \ref{2}, $\{f_n\}$ is a higher $\{g_n, h_n\}$-derivation on $\mathcal{A}$ with $f_0 = g_0 = h_0 = I$. This yields the desired result.
\end{proof}

An immediate corollary of the precede theorem reads as follows:

\begin{corollary}\label{+} If $\{d_n\}_{n = 0, 1, ...}$ with $d_0 = I$ is a higher derivation on $\mathcal{A}$, then there is a sequence $\{\delta_n\}_{n = 0, 1, ...}$ of derivations with $\delta_0 = 0$ characterizing the higher derivation $\{d_n\}_{n = 0, 1, ...}$. As well as if $\{\delta_n\}_{n = 0, 1, ...}$ is a sequence of derivations with $\delta_0 = 0$, then there exists a higher derivation $\{d_n\}$ with $d_0 = I$ which is characterized by the sequence $\{\delta_n\}_{n = 0, 1, ...}$.
\end{corollary}

\vspace{.25cm}


\bibliographystyle{amsplain}

\end{document}